\documentclass{article}

\usepackage{amsmath}
\usepackage{amssymb}
\usepackage{graphicx}
\usepackage{bm}
\usepackage{mathdots}
\usepackage{booktabs}
\usepackage{rotating}
\usepackage[ruled,linesnumbered]{algorithm2e}
\usepackage[a4paper,left=2.8cm,right=2.8cm,top=2.5cm,bottom=2.5cm]{geometry}
\usepackage{fancyhdr}
\usepackage[framemethod=tikz]{mdframed}
\newmdtheoremenv[%
backgroundcolor=green!10,%
outerlinecolor=black,%
leftmargin=0,%
rightmargin=0,
innertopmargin =3pt,%
innerleftmargin = 5pt,
innerrightmargin = 5pt,
splittopskip = \topskip,%
skipabove = \baselineskip,%
skipbelow = \baselineskip,%
roundcorner=5, ntheorem]
{theorem}{Theorem}[section]
\newtheorem{corollary}{Corollary}[section]

\newtheorem{lemma}{Lemma}[section]
\newtheorem{definition}{Definition}[section]
\newtheorem{assumption}{Assumption}[section]

\makeatletter 
\@addtoreset{equation}{section}
\makeatother  

\newtheorem{remark}{Remark}[section]

\newenvironment{proof}{{\noindent\it Proof.}\quad}{\hfill $\square$\\}

\begin{document}
\title{On the quadrature exactness in hyperinterpolation}

\author{Congpei An\footnotemark[1]
       \quad\quad Hao-Ning Wu\footnotemark[2]\\~\\

       \emph{Dedicated to Ian H. Sloan on the occasion of his 85th birthday.}}

\renewcommand{\thefootnote}{\fnsymbol{footnote}}
\footnotetext[1]{School of Mathematics, Southwestern University of Finance and Economics, Chengdu, China (ancp@swufe.edu.cn, andbachcp@gmail.com).}
\footnotetext[2]{Department of Mathematics, The University of Hong Kong, Hong Kong, China (hnwu@connect.hku.hk).}


\maketitle

\begin{abstract}
 This paper investigates the role of quadrature exactness in the approximation scheme of hyperinterpolation. Constructing a hyperinterpolant of degree $n$ requires a positive-weight quadrature rule with exactness degree $2n$. We examine the behavior of such approximation when the required exactness degree $2n$ is relaxed to $n+k$ with $0<k\leq n$. Aided by the Marcinkiewicz--Zygmund inequality, we affirm that the $L^2$ norm of the  exactness-relaxing hyperinterpolation operator is bounded by a constant independent of $n$, and this approximation scheme is convergent as $n\rightarrow\infty$ if $k$ is positively correlated to $n$. Thus, the family of candidate quadrature rules for constructing hyperinterpolants can be significantly enriched, and the number of quadrature points can be considerably reduced. As a potential cost, this relaxation may slow the convergence rate of hyperinterpolation in terms of the reduced degrees of quadrature exactness. Our theoretical results are asserted by numerical experiments on three of the best-known quadrature rules: the Gauss quadrature, the Clenshaw--Curtis quadrature, and the spherical $t$-designs.
\end{abstract}



\section{Introduction}

Let $\Omega$ be a bounded region of $\mathbb{R}^s$ with measure $\text{d}\omega$, which is either the closure of a connected open domain, or a smooth closed lower-dimensional manifold in $\mathbb{R}^s$. This region is assumed to have finite measure with respect to $\text{d}\omega$, that is,
$$\int_{\Omega}\text{d}\omega=V<\infty. $$
We denote by $\mathbb{P}_n\subset L^2(\Omega)$ the linear space of polynomials on $\Omega$ of degree at most $n$, equipped with the $L^2$ inner product
\begin{equation}\label{equ:innerproduct}
\langle v,z\rangle = \int_{\Omega}vz\text{d}\omega,
\end{equation}
and we let $\{p_1,p_2\ldots,p_{d_n}\}\subset\mathbb{P}_n$ be an orthonormal basis of $\mathbb{P}_n$ in the sense of $\langle p_{\ell},p_{\ell'}\rangle = \delta_{\ell\ell'}$ for $1\leq \ell,\ell'\leq d_n$, where $d_n=\dim\mathbb{P}_n$ is the dimension of $\mathbb{P}_n$ . Constructing hyperinterpolants requires an $m$-point quadrature rule of the form
\begin{equation}\label{equ:quad}
\sum_{j=1}^mw_jg(x_j)\approx \int_{\Omega}g\text{d}\omega,
\end{equation}
where the quadrature points $x_j$ belong to $\Omega$ and weights $w_j$ are all positive for $j=1,2,\ldots,m$; we refer the reader to the classic book \cite{MR760629} for a comprehensive introduction in numerical integration. With the assumption that the quadrature rule \eqref{equ:quad} has exactness degree $2n$, i.e.,
$$\sum_{j=1}^mw_jg(x_j) = \int_{\Omega}g\text{d}\omega\quad \forall g\in\mathbb{P}_{2n},$$
the hyperinterpolation operator $\mathcal{L}_n:\mathcal{C}(\Omega)\rightarrow\mathbb{P}_n$, introduced by Sloan in \cite{sloan1995polynomial}, maps a continuous function $f\in\mathcal{C}(\Omega)$ on $\Omega$ to
\begin{equation}\label{equ:hyperinterpolation}
\mathcal{L}_nf:=\sum_{\ell=1}^{d_n}\langle f,p_{\ell}\rangle_mp_{\ell},
\end{equation}
where
$$\langle v,z\rangle_m:=\sum_{j=1}^mw_jv(x_j)z(x_j)$$
is a ``discrete version'' of the $L^2$ inner product \eqref{equ:innerproduct}. Thus, hyperinterpolation can be regarded as a discrete version of the orthogonal projection from $\mathcal{C}(\Omega)$ onto $\mathbb{P}_n$ with respect to \eqref{equ:innerproduct}.

The bulk of the subsequent development on hyperinterpolation was on the sphere, see \cite{dai2006generalized,MR2274179,le2001uniform,reimer2002generalized,zbMATH01421286}. Hyperinterpolation was also investigated on many other regions, such as the disk \cite{hansen2009norm}, the square \cite{caliari2007hyperinterpolation}, the cube \cite{caliari2008hyperinterpolation,wang2014norm}, and the spherical triangles \cite{sommariva2021numerical}. In all of these references, the exactness degree $2n$ of the quadrature rule \eqref{equ:quad} is a central assumption in constructing hyperinterpolants. This assumption was also maintained in some variants of hyperinterpolation, such as the filtered hyperinterpolation \cite{sloan2012filtered} (even more degrees are required) and the Lasso hyperinterpolation \cite{an2021lasso}.

Moreover, if one considers hyperinterpolation on some regions where quadrature theory has not been well established, this exactness assumption has also potentially spurred the development of quadrature theory and orthogonal polynomials on these regions. Indeed, quadrature exactness contributes to the standard principle for designing quadrature rules: they should be exact for a certain class of integrands, e.g., polynomials under a fixed degree. This exactness principle is the departing point of most discussions on quadrature. Still, there has been growing concern recently about whether this principle is reliable in designing quadrature rules, as discussed by Trefethen in \cite{trefethen2022exactness}. The main message of \cite{trefethen2022exactness} is that the exactness principle proves to be an unreliable guide to actual accuracy. According to Trefethen, the exactness principle is a matter of algebra, concerned with whether or not certain quantities are exactly zero; however, quadrature is a problem of analysis, focusing on whether or not certain quantities are small. Thus, we are intrigued to know whether the required exactness degree $2n$ in constructing hyperinterpolants of degree $n$ is superfluous.

This question is answered as the main results of this paper: When $2n$ is relaxed to $n+k$, where $0<k\leq n$, i.e., reduced at least to $n+1$, the norm of $\mathcal{L}_n$ as an operator from $\mathcal{C}(\Omega)$ to $L^2(\Omega)$ is bounded by some constant, and the error estimate $\|\mathcal{L}_nf-f\|_2$ is bounded in terms of $E_k(f)$, which is the best uniform error of $f$ by a polynomial in $\mathbb{P}_k$. In addition, if $k$ is positively correlated to $n$, then the scheme of hyperinterpolation is convergent as $n\rightarrow \infty$. This relaxation helps hyperinterpolation to get rid of the disadvantage that, remarked by Hesse and Sloan in \cite{MR2274179}, it needs function values at the given points of the positive-weight quadrature rule with exactness degree $2n$. In real-world applications, data sampling may be expensive. This relaxation may enlighten us to develop hyperinterpolation-based methods for problems that are in favor of a high-order approximation but against extensive data sampling. When data sampling is cheap, this relaxation may also help to speed up our computation.

We note that the generalized hyperinterpolation \cite{dai2006generalized,reimer2002generalized}, defined on the sphere, only requires a positive-weight quadrature rule with exactness degree $n+1$ rather than $2n$. However, the definition of this scheme is different from that of the original hyperinterpolation. In this paper, we focus on the original hyperinterpolation and investigate the effects of relaxing the quadrature exactness. Moreover, our investigation pertains to a general region $\Omega$, while the generalized hyperinterpolation is only studied on the sphere.

In the next section, we present the main theoretical results on the exactness-relaxing hyperinterpolation, with the proof of our main Theorem \ref{thm} given in Section \ref{sec:proof}. To verify our theory, we conduct some numerical experiments on the interval $[-1,1]$ and the unit sphere $\mathbb{S}^2$ in Section \ref{sec:3}.

\section{Main results}
The hyperinterpolant of degree $n$ with an exactness-relaxing quadrature rule is defined as follows.

\begin{assumption}\label{assumption}
The $m$-point quadrature rule \eqref{equ:quad}, with nodes $x_j\in\Omega$ and weights $w_j>0$ for $j=1,2,\ldots,m$, has exactness degree $n+k$ with $0<k\leq n$, where $n,k\in\mathbb{N}$.
\end{assumption}
\begin{definition}[Hyperinterpolation with an exactness-relaxing quadrature rule]\label{def}
Let $\langle \cdot,\cdot\rangle_m$ be an $m$-point quadrature rule fulfilling Assumption \ref{assumption} and $\{p_{\ell}\}_{\ell=1}^{d_n}\subset\mathbb{P}_n$ be an orthonormal basis of $\mathbb{P}_n$. Given $f\in\mathcal{C}(\Omega)$, the hyperinterpolant of degree $n$ to $f$ is defined as
\begin{equation}\label{equ:newhyper}
\mathcal{L}_nf:=\sum_{\ell=1}^{d_n}\langle f,p_{\ell}\rangle_mp_{\ell}.
\end{equation}
\end{definition}
This scheme \eqref{equ:newhyper} is essentially the hyperinterpolation scheme \eqref{equ:hyperinterpolation}, except that the degree of quadrature exactness is relaxed. Thus the scheme \eqref{equ:newhyper} is also a discrete version of the orthogonal projection from $\mathcal{C}(\Omega)$ onto $\mathbb{P}_n$ with respect to the $L^2$ inner product \eqref{equ:innerproduct}. To tell the difference between schemes \eqref{equ:hyperinterpolation} and \eqref{equ:newhyper}, we refer to Sloan's hyperinterpolation as the \emph{original hyperinterpolation}. We denote by $\mathcal{L}_n^{\text{S}}$ the original hyperinterpolation operator in the following texts, where \emph{S} stands for Sloan.

What kind of benefits and costs does the relaxation of quadrature exactness bring to the analysis and implementation of hyperinterpolation? Here is an immediate benefit. We know that an $m$-point quadrature rule with exactness degree $2n$ requires $m\geq d_n$ quadrature points, see \cite[Lemma 2]{sloan1995polynomial}, and such a quadrature rule is said to be \emph{minimal} if $m=d_n$. This fact suggests that $m$ should satisfy $m\geq d_n$ for $\mathcal{L}^{\text{S}}_n$, and it also admits the following rather simple but interesting theorem.
\begin{theorem}\label{thm:number}
The number of quadrature points for the hyperinterpolation \eqref{equ:newhyper} satisfies
\begin{equation*}
m\geq \begin{cases}
d_{(n+k)/{2}}= \dim\mathbb{P}_{(n+k)/2},& \text{ when }n+k \text{ is even},\\
d_{(n+k+1)/{2}}= \dim\mathbb{P}_{(n+k+1)/{2}},& \text{ when }n+k \text{ is odd}.
\end{cases}
\end{equation*}
\end{theorem}
The benefit brought by the theorem is two-fold. On the one hand, for minimal quadrature rules used in constructing hyperinterpolants, the required amount of quadrature points can be considerably reduced from $\dim\mathbb{P}_{n}$ to $\dim\mathbb{P}_{(n+k)/{2}}$ or $\dim\mathbb{P}_{(n+k+1)/{2}}$, depending on the parity of $n+k$. Such reduction is more pronounced in higher-dimensional regions. On the other hand, for quadrature rules demanding more nodes to achieve the exactness degree $2n$, which used to be deemed impractical, some of them can be added into the family of candidate quadrature rules to construct hyperinterpolants efficiently. For example, a typical choice of quadrature rules for hyperinterpolation on $[-1,1]$ is the Gaussian quadrature, and now the Clenshaw--Curtis quadrature can also be considered a good choice; see more details in Section \ref{sec:3}.

Obviously, such relaxation is not cost-free. The original hyperinterpolant \eqref{equ:hyperinterpolation} is a projection for $f\in\mathbb{P}_n$, that is, $\mathcal{L}^{\text{S}}_nf=f$ for all $f\in\mathbb{P}_n$; see \cite[Lemma 4]{sloan1995polynomial}. However, due to the loss of some exactness degrees, this property is preserved only for polynomials of degree at most $k$, asserted by the following lemma.
\begin{lemma}\label{lem:polynomial}
If $f\in\mathbb{P}_k$, then $\mathcal{L}_n$ defined in Definition \ref{def} admits $\mathcal{L}_nf=f$.
\end{lemma}
\begin{proof}
For $f\in\mathbb{P}_k$, it may be expressed as $f=\sum_{\ell=1}^{d_k}a_{\ell}p_{\ell}$, where $a_{\ell} = \int_{\Omega}fp_{\ell}\text{d}\omega$ and $d_k=\dim\mathbb{P}_k$. The exactness degree $n+k$ admits $\langle p_{\ell'},p_{\ell}\rangle_m =\delta_{\ell\ell'}$ for $1\leq \ell'\leq d_k$ and $1\leq \ell\leq d_n$. Thus,
$$\mathcal{L}_nf=\sum_{\ell=1}^{d_n}\left\langle \sum_{\ell'=1}^{d_k}a_{\ell'}p_{\ell'}, p_{\ell}\right\rangle_m p_{\ell}
=\sum_{\ell=1}^{d_n}\left(\sum_{\ell'=1}^{d_k}a_{\ell'} \left\langle p_{\ell'},p_{\ell}\right \rangle_m\right) p_{\ell}=\sum_{\ell=1}^{d_k}a_{\ell}p_{\ell},$$
leading to $\mathcal{L}_nf=f$.
\end{proof}

\begin{corollary}
For $f\in\mathcal{C}(\Omega)$, we have $\mathcal{L}_n(\mathcal{L}_kf)=\mathcal{L}_k(\mathcal{L}_nf)=\mathcal{L}_k(\mathcal{L}_kf)=\mathcal{L}_kf$.
\end{corollary}
\begin{proof}
As $\mathcal{L}_kf\in\mathbb{P}_k$, Lemma \ref{lem:polynomial} immediately implies $\mathcal{L}_n(\mathcal{L}_kf) = \mathcal{L}_kf$.

Similar to the proof of Lemma \ref{lem:polynomial}, we have
\begin{equation*}\begin{split}
\mathcal{L}_k(\mathcal{L}_nf)=&\sum_{\ell=1}^{d_k}\left\langle \sum_{\ell'=1}^{d_n} \langle f,p_{\ell'}\rangle_mp_{\ell'},p_{\ell}\right\rangle_mp_{\ell}
= \sum_{\ell=1}^{d_k}\left( \sum_{\ell'=1}^{d_n} \langle f,p_{\ell'}\rangle_m\langle p_{\ell'},p_{\ell}\rangle_m\right)p_{\ell}\\
&= \sum_{\ell=1}^{d_k}\langle f,p_{\ell}\rangle_mp_{\ell} = \mathcal{L}_kf,
\end{split}\end{equation*}
and similarly,
\begin{equation*}\begin{split}
\mathcal{L}_k(\mathcal{L}_kf)= \sum_{\ell=1}^{d_k}\left( \sum_{\ell'=1}^{d_k} \langle f,p_{\ell'}\rangle_m\langle p_{\ell'},p_{\ell}\rangle_m\right)p_{\ell}= \sum_{\ell=1}^{d_k}\langle f,p_{\ell}\rangle_mp_{\ell} = \mathcal{L}_kf.
\end{split}\end{equation*}
Thus, the corollary is completely proved.
\end{proof}

\begin{remark}
Lemma \ref{lem:polynomial} indicates that the exactness degree $2n$ can be relaxed at least to $n+1$; otherwise, the projection property $\mathcal{L}_nf=f$ for all $f\in\mathbb{P}_k$ does not maintain for any non-trivial polynomial spaces.
\end{remark}
\begin{remark}
There may be an illusion that for the exactness-relaxing hyperinterpolation \eqref{equ:newhyper}, there holds $\mathcal{L}_nf=f$ for $f\in\mathbb{P}_{\lfloor (n+k)/{2}\rfloor}$, induced from the fact that for $\mathcal{L}^{\rm{S}}_n$ with exactness degree $2n$, $\mathcal{L}^{\rm{S}}_nf=f$ for all $f\in\mathbb{P}_{n}$. However, according to the proof of Lemma \ref{lem:polynomial}, this is not true. Indeed, $\langle p_{\ell'},p_{\ell}\rangle_m$ with exactness degree $n+k$ may not be the Kronecker $\delta_{\ell\ell'}$ for $p_{\ell'}\in\mathbb{P}_{{\lfloor (n+k)/{2}\rfloor}}$ and $p_{\ell}\in\mathbb{P}_n$.
\end{remark}

This decay of projection-maintaining degrees is followed by Theorem \ref{thm} below, indicating that the convergence rate of $\mathcal{L}^{\text{S}}_n$ is slowed from $E_n(f)$ to $E_k(f)$. It was proved in \cite{sloan1995polynomial} that
\begin{equation}\label{equ:stabilityoriginal}
\|\mathcal{L}^{\text{S}}_nf\|_2\leq V^{1/2}\|f\|_{\infty}
\end{equation}
and
\begin{equation}\label{equ:errororiginal}
\|\mathcal{L}^{\text{S}}_nf-f\|_2\leq 2V^{1/2}E_n(f),
\end{equation}
where the appeared norms are defined as $\|g\|_2:=\left(\int_{\Omega}|g|^2\text{d}\omega\right)^{1/2}$ for $g\in L^2(\Omega)$ and $\|g\|_{\infty}:=\sup_{x\in\Omega}|g(x)|$ for $g\in\mathcal{C}(\Omega)$, and $E_n(g)$ denotes the best uniform error of $g$ by a polynomial in $\mathbb{P}_n$, that is,
$$E_n(g):=\inf_{\chi\in \mathbb{P}_n}\|g-\chi\|_{\infty}\quad \forall g\in\mathcal{C}(\Omega).$$

To tell the difference between the stability result \eqref{equ:stabilityoriginal} of $\mathcal{L}^{\text{S}}_n$ and that of $\mathcal{L}_n$, we note that the stability result \eqref{equ:stabilityoriginal} stems from
$$\|\mathcal{L}^{\text{S}}_nf\|_2^2+\langle f- \mathcal{L}^{\text{S}}_nf, f- \mathcal{L}^{\text{S}}_nf\rangle_{m'} = \langle f,f \rangle_{m'}=\sum_{j=1}^mw_jf(x_j)^2\leq V\|f\|_{\infty}^2 $$
and the non-negativeness of $\langle f- \mathcal{L}^{\text{S}}_nf, f- \mathcal{L}^{\text{S}}_nf\rangle_{m'}$, where $\langle\cdot,\cdot\rangle_{m'}$ denotes an $m$-point quadrature rule \eqref{equ:quad} with exactness degree $2n$ and this notation is only used here; see the proof in \cite{sloan1995polynomial}. However, due to the relaxation of exactness degrees, we can only claim
\begin{equation*}\begin{split}
\|\mathcal{L}_nf\|_2^2+\langle f- \mathcal{L}_nf, f- \mathcal{L}_nf\rangle_m + \sigma_{n,k,f}= \langle f,f \rangle_m,
\end{split}\end{equation*}
where
\begin{equation}\label{equ:sigma}
\sigma_{n,k,f}=\langle\mathcal{L}_nf-\mathcal{L}_kf,\mathcal{L}_nf-\mathcal{L}_kf\rangle-\langle\mathcal{L}_nf-\mathcal{L}_kf,\mathcal{L}_nf-\mathcal{L}_kf\rangle_m
\end{equation}
stands for the error in evaluating the integral of $(\mathcal{L}_nf-\mathcal{L}_kf)^2$ over $\Omega$ by the quadrature rule \eqref{equ:quad} with exactness degree $n+k$; see the equation \eqref{equ:equality} in our proof in the next section. Even though it is possible (and often occurs) that $\langle f- \mathcal{L}_nf, f- \mathcal{L}_nf\rangle_m + \sigma_{n,k,f}\geq 0$ if the quadrature rule \eqref{equ:quad} converges fast enough, we cannot make such a claim rigorously in general. Therefore, it is natural to endow the quadrature rule \eqref{equ:quad} with some convergence property.

We assume that there exists an $\eta\in[0,1)$ such that
\begin{equation}\label{equ:etaassumption}
\left|\sum_{j=1}^mw_j\chi(x_j)^2-\int_{\Omega}\chi^2\text{d}\omega\right|\leq \eta \int_{\Omega}\chi^2\text{d}\omega\quad \forall \chi\in\mathbb{P}_{n}.
\end{equation}
If $k=n$, i.e., the quadrature exactness is not relaxed, then $\eta =0$. This convergence property \eqref{equ:etaassumption} can be regarded as the Marcinkiewicz--Zygmund inequality \cite{filbir2011marcinkiewicz,Marcinkiewicz1937,mhaskar2001spherical} applied to polynomials of degree at most $2n$, and we refer to it as the \emph{Marcinkiewicz--Zygmund property} below.

\begin{theorem}\label{thm}
Given $f\in\mathcal{C}(\Omega)$, let $\mathcal{L}_nf\in\mathbb{P}_n$ be defined by \eqref{equ:newhyper}, where the $m$-point quadrature rule \eqref{equ:quad} not only fulfills Assumption \ref{assumption} with $0<k\leq n$ but also has the Marcinkiewicz--Zygmund property \eqref{equ:etaassumption} with $\eta\in[0,1)$. Then
\begin{equation}\label{equ:stability}
\|\mathcal{L}_nf\|_2\leq\frac{V^{1/2}}{\sqrt{1-\eta}}\|f\|_{\infty},
\end{equation}
and
\begin{equation}\label{equ:error}
\|\mathcal{L}_nf-f\|_2\leq \left(\frac{1}{\sqrt{1-\eta}}+1\right)V^{1/2}E_k(f).
\end{equation}
The hyperinterpolant $\mathcal{L}_nf$ may not converge to $f$ as $n\rightarrow\infty$ if $k$ is fixed. If $k$ is additionally positively correlated to $n$, then
\begin{equation}\label{equ:conv}
\|\mathcal{L}_nf-f\|_2\rightarrow 0\quad \text{as}\quad n\rightarrow \infty.
\end{equation}

\end{theorem}

\begin{remark}
By ``$k$ is additionally positively correlated to $n$,'' we mean that $n\rightarrow \infty$ implies $k\rightarrow \infty$. This condition ensures the convergence result \eqref{equ:conv} as $n\rightarrow \infty$. The converse statement that $k\rightarrow \infty$ implies $n\rightarrow \infty$ automatically holds because $k\leq n$.
\end{remark}

\begin{remark}
If $k=n$, i.e., the degree of quadrature exactness is not relaxed, then the stability result \eqref{equ:stability}, the error estimate \eqref{equ:error}, and the convergence result \eqref{equ:conv} are the same as those for $\mathcal{L}^{\rm{S}}_n$ in \cite{sloan1995polynomial}. If $0<k<n$, then as a cost of the relaxation of exactness, the error estimation \eqref{equ:error} is now controlled by $E_k(f)$ rather than $E_n(f)$. Since $E_k(f) \geq E_n(f)$ if $k<n$, this estimation \eqref{equ:error} reveals an effect of relaxing the quadrature exactness. That is, we can use fewer quadrature points than the original hyperinterpolation, but the corresponding error estimation will be somewhat amplified. Moreover, if $k\leq 0$, i.e., the degree of quadrature exactness is relaxed to $n$ or even less, then no convergence information can be offered by Theorem \ref{thm}.
\end{remark}



An immediate application of Theorem \ref{thm} is to a generalization of the method of ``product integration'',  see discussions in \cite{sloan1995polynomial}. In this method, the integral over $\Omega$ of the form $\int_{\Omega}hf\text{d}\omega$, where $f$ is smooth and $h$ contains any singularities in the product integrand, is approximated by
\begin{equation}\label{equ:productintegration}
\int_{\Omega}hf\text{d}\omega\approx\int_{\Omega}h(\mathcal{L}_nf)\text{d}\omega = \sum_{\ell=1}^{d_n}\langle f,p_{\ell}\rangle_m\int_{\Omega} h p_{\ell}\text{d}\omega=\sum_{j=1}^m W_jf(x_j),
\end{equation}
where
\begin{equation}\label{equ:WJ}
W_j = w_j \sum_{\ell=1}^{d_n} p_{\ell}(x_j)\int_{\Omega} h p_{\ell}\text{d}\omega,\quad j = 1,2,\ldots,m.
\end{equation}
Applying the Cauchy--Schwarz inequality over $\Omega$ to $\int_{\Omega}h(\mathcal{L}_nf-f)\text{d}\omega$, Theorem \ref{thm} immediately implies the following result.
\begin{corollary}
Let $h$ be measurable on $\Omega$ with respect to ${\rm{d}}\omega$ and satisfy $\|h\|_2<\infty$, and let $\{W_j\}_{j=1}^m$ be given by \eqref{equ:WJ}. Under the conditions of Theorem \ref{thm}, the approximation error of $\int_{\Omega}hf{\rm{d}}\omega$ in terms of \eqref{equ:productintegration} is estimated by
$$\left|\sum_{j=1}^m W_jf(x_j)-\int_{\Omega}hf{\rm{d}}\omega\right| \leq \left(\frac{1}{\sqrt{1-\eta}}+1\right)\|h\|_2V^{1/2}E_k(f).$$
\end{corollary}

\begin{remark}
In the light of Theorem \ref{thm}, we expect that the required exactness degree in constructing other variants of hyperinterpolants, such as filtered hyperinterpolants \cite{sloan2012filtered} and Lasso hyperinterpolants \cite{an2021lasso}, can also be reduced, and corresponding theory can be developed.
\end{remark}

\section{Proof of Theorem \ref{thm}}\label{sec:proof}

\subsection{Preparation}
The hyperinterpolant $\mathcal{L}_nf$ can be decomposed into
\begin{equation}\label{equ:decomposition}
\mathcal{L}_nf:=\mathcal{L}_kf+(\mathcal{L}_n-\mathcal{L}_k)f,
\end{equation}
where $\mathcal{L}_n-\mathcal{L}_k: \mathcal{C}(\Omega)\rightarrow \mathbb{P}_n$ is a linear operator mapping $f\in\mathcal{C}(\Omega)$ to
$$(\mathcal{L}_n-\mathcal{L}_k)f:=\sum_{\ell=d_k+1}^{d_n}\langle f,p_{\ell}\rangle_m p_{\ell} \in\mathbb{P}_n.$$
In the following proof of Theorem \ref{thm}, we shall treat $\mathcal{L}_kf$ and $(\mathcal{L}_n-\mathcal{L}_k)f$ separately. For the former component, the degree $n+k\geq 2k$ of quadrature exactness leads to
\begin{equation}\label{equ:observation1}
\langle \mathcal{L}_kf,\mathcal{L}_kf\rangle=\langle \mathcal{L}_kf,\mathcal{L}_kf\rangle_m.\
\end{equation}
For the latter component, the orthogonality of $\{p_{\ell}\}$ renders
\begin{equation}\label{equ:observation2}
\langle (\mathcal{L}_n-\mathcal{L}_k)f,(\mathcal{L}_n-\mathcal{L}_k)f\rangle=\sum_{\ell=d_k+1}^{d_n}\langle f,p_{\ell}\rangle_m^2=\langle f,(\mathcal{L}_n-\mathcal{L}_k)f\rangle_m.
\end{equation}

Before proving Theorem \ref{thm}, we present a lemma involving $\langle \mathcal{L}_kf,\mathcal{L}_kf\rangle_m$ and $\langle f,(\mathcal{L}_n-\mathcal{L}_k)f\rangle_m$.
\begin{lemma}\label{lem}
Adopt the conditions of Theorem \ref{thm}. Let $\mathcal{L}_k:\mathcal{C}(\Omega)\rightarrow \mathbb{P}_k$ be the hyperinterpolation operator of degree $k$, defined with an $m$-point quadrature with exactness degree $n+k$. Then

{\rm{(a)}} $\langle f-\mathcal{L}_kf,\chi\rangle_m = 0$ and $\langle f-\mathcal{L}_nf,\chi\rangle_m = 0$ for all $\chi\in\mathbb{P}_k$,

{\rm{(b)}} $\langle \mathcal{L}_kf,\mathcal{L}_kf\rangle_m + \langle f-\mathcal{L}_kf,f-\mathcal{L}_kf\rangle_m = \langle f,f\rangle_m $,

{\rm{(c)}} $\langle \mathcal{L}_kf,\mathcal{L}_kf\rangle_m + \langle \mathcal{L}_nf-\mathcal{L}_kf,\mathcal{L}_nf-\mathcal{L}_kf\rangle_m = \langle \mathcal{L}_nf,\mathcal{L}_nf\rangle_m $,

{\rm{(d)}} $\langle f-\mathcal{L}_nf,f-\mathcal{L}_nf\rangle_m + 2\langle f,\mathcal{L}_nf-\mathcal{L}_kf  \rangle_m = \langle f-\mathcal{L}_kf,f-\mathcal{L}_kf\rangle_m+ \langle \mathcal{L}_nf-\mathcal{L}_kf,\mathcal{L}_nf-\mathcal{L}_kf\rangle_m$.

\end{lemma}

\begin{proof}
(a) Note that any $\chi\in\mathbb{P}_k$ can be expressed as $\chi=\sum_{\ell=1}^{d_k}a_{\ell}p_{\ell}$, where $a_{\ell}=\int_{\Omega} \chi p_{\ell} \text{d}\omega$. The first equation holds since
\begin{equation*}\begin{split}
\langle f-\mathcal{L}_kf,\chi\rangle_m &= \sum_{\ell=1}^{d_k}a_{\ell}\left\langle f-\sum_{\ell'=1}^{d_k}\langle f,p_{\ell'}\rangle_m p_{\ell'},p_{\ell}\right\rangle_m \\
&= \sum_{\ell=1}^{d_k}a_{\ell}\left(\langle f,p_{\ell}\rangle_m -\sum_{\ell'=1}^{d_k}\langle f,p_{\ell'}\rangle_m \langle p_{\ell'},p_{\ell}\rangle_m\right)=0.
\end{split}\end{equation*}
Similarly,
\begin{equation*}\begin{split}
\langle f-\mathcal{L}_nf,\chi\rangle_m = \sum_{\ell=1}^{d_k}a_{\ell}\left(\langle f,p_{\ell}\rangle_m -\sum_{\ell'=1}^{d_n}\langle f,p_{\ell'}\rangle_m \langle p_{\ell'},p_{\ell}\rangle_m\right)=0.
\end{split}\end{equation*}

(b) Letting $\chi = \mathcal{L}_kf$, the first equation in statement (a) implies $\langle \mathcal{L}_kf,\mathcal{L}_kf\rangle_m = \langle f,\mathcal{L}_kf\rangle_m$. Thus
\begin{equation*}\begin{split}
\langle \mathcal{L}_kf,\mathcal{L}_kf\rangle_m + \langle f-\mathcal{L}_kf,f-\mathcal{L}_kf\rangle_m  &= 2\langle \mathcal{L}_kf,\mathcal{L}_kf\rangle_m
-2 \langle f,\mathcal{L}_kf\rangle_m +  \langle f,f\rangle_m \\
&= \langle f,f\rangle_m.
\end{split}
\end{equation*}

(c) Letting $\chi = \mathcal{L}_kf$ in both equations in statement (a), we have $\langle \mathcal{L}_kf,\mathcal{L}_kf\rangle_m = \langle f,\mathcal{L}_kf\rangle_m = \langle \mathcal{L}_nf,\mathcal{L}_kf\rangle_m $. Thus
\begin{equation*}\begin{split}
&\langle \mathcal{L}_kf,\mathcal{L}_kf\rangle_m + \langle \mathcal{L}_nf-\mathcal{L}_kf,\mathcal{L}_nf-\mathcal{L}_kf\rangle_m \\
 =& 2\langle \mathcal{L}_kf,\mathcal{L}_kf\rangle_m-2 \langle \mathcal{L}_nf,\mathcal{L}_kf\rangle_m +  \langle \mathcal{L}_nf,\mathcal{L}_nf\rangle_m \\
= &\langle \mathcal{L}_nf,\mathcal{L}_nf\rangle_m.
\end{split}
\end{equation*}

(d) It is immediate that
\begin{equation}\label{equ:g}
\langle g-\mathcal{L}_ng,g-\mathcal{L}_ng\rangle_m = \langle g,g\rangle_m -2\langle g,\mathcal{L}_ng\rangle_m+\langle \mathcal{L}_ng,\mathcal{L}_ng\rangle_m
\end{equation}
holds for any $g\in\mathcal{C}(\Omega)$. Lemma \ref{lem:polynomial} implies $\mathcal{L}_n(\mathcal{L}_kf)=\mathcal{L}_kf$. Then replacing $g$ by $f-\mathcal{L}_kf$, the left-hand side of \eqref{equ:g} becomes
\begin{equation*}\begin{split}
&\langle f-\mathcal{L}_kf-\mathcal{L}_n(f-\mathcal{L}_kf),f-\mathcal{L}_kf-\mathcal{L}_n(f-\mathcal{L}_kf)\rangle_m\\
=&\langle f-\mathcal{L}_kf -\mathcal{L}_nf+\mathcal{L}_kf,f-\mathcal{L}_kf-\mathcal{L}_nf+\mathcal{L}_kf\rangle_m\\
=& \langle f-\mathcal{L}_nf,f-\mathcal{L}_nf\rangle_m,
\end{split}\end{equation*}
and three terms on the right-hand side becomes $\langle g,g\rangle_m  = \langle f-\mathcal{L}_kf,f-\mathcal{L}_kf\rangle_m$,
\begin{align}
-2\langle g,\mathcal{L}_ng\rangle_m =& -2\langle f-\mathcal{L}_kf,\mathcal{L}_n(f-\mathcal{L}_kf)\rangle_m \notag\\
 =& -2\langle f-\mathcal{L}_kf,\mathcal{L}_nf-\mathcal{L}_kf\rangle_m \notag\\
=&-2\langle f,\mathcal{L}_nf-\mathcal{L}_kf\rangle_m,\label{equ:term2}
\end{align}
and
\begin{equation*}\begin{split}
\langle \mathcal{L}_ng,\mathcal{L}_ng\rangle_m
=& \langle \mathcal{L}_n(f-\mathcal{L}_kf),\mathcal{L}_n(f-\mathcal{L}_kf)\rangle_m \\
=&\langle \mathcal{L}_nf-\mathcal{L}_kf,\mathcal{L}_nf-\mathcal{L}_kf\rangle_m,
\end{split}\end{equation*}
where \eqref{equ:term2} holds since the orthogonality of $\{p_{\ell}\}_{\ell=1}^{d_n}$ and the quadrature exactness degree $n+k$ imply $\langle p_{\ell},p_{\ell'}\rangle_m = 0$ for $\ell = 1,2,\ldots,d_k$ and $\ell'= d_k+1,\ldots,d_n$, and then
\begin{equation}\label{equ:proofuse}
\langle \mathcal{L}_kf,(\mathcal{L}_n-\mathcal{L}_k)f\rangle_m=\left\langle \sum_{\ell=1}^{d_k}\langle f,p_{\ell}\rangle_m p_{\ell},\sum_{\ell'=d_k+1}^{d_n} \langle f,p_{\ell'}\rangle_m p_{\ell'}\right\rangle_m=0.
\end{equation}
Hence, the equality \eqref{equ:g} suggests the proof of statement (d).
\end{proof}

\subsection{Proof of Theorem \ref{thm}}
Now we are prepared to prove Theorem \ref{thm}.

\noindent \emph{Proof of Theorem \ref{thm}.} According to the decomposition \eqref{equ:decomposition}, we have
\begin{equation*}\begin{split}
\|\mathcal{L}_nf\|_2^2 = &\langle \mathcal{L}_nf,\mathcal{L}_nf\rangle =\langle \mathcal{L}_kf+(\mathcal{L}_n-\mathcal{L}_k)f,\mathcal{L}_kf+(\mathcal{L}_n-\mathcal{L}_k)f\rangle \\
=&\langle \mathcal{L}_kf,\mathcal{L}_kf\rangle + \langle (\mathcal{L}_n-\mathcal{L}_k)f,(\mathcal{L}_n-\mathcal{L}_k)f\rangle,
\end{split}\end{equation*}
where the last step holds since $\langle \mathcal{L}_kf,(\mathcal{L}_n-\mathcal{L}_k)f\rangle=0$, which can be proved similarly to \eqref{equ:proofuse} and using the fact that $\langle p_{\ell},p_{\ell'}\rangle = 0$ for $\ell = 1,2,\ldots,d_k$ and $\ell'= d_k+1,\ldots,d_n$. The observations \eqref{equ:observation1} and \eqref{equ:observation2} then lead to
\begin{equation*}\begin{split}
\|\mathcal{L}_nf\|_2^2 = \langle \mathcal{L}_kf,\mathcal{L}_kf\rangle_m + \langle f,(\mathcal{L}_n-\mathcal{L}_k)f\rangle_m.
\end{split}\end{equation*}

To derive the stability result \eqref{equ:stability}, summing up the equations in Lemma \ref{lem}(b,c,d), after easy computations, we have
\begin{equation}\label{equ:summingup}\begin{split}
&2\langle \mathcal{L}_kf,\mathcal{L}_kf\rangle_m+2 \langle f,(\mathcal{L}_n-\mathcal{L}_k)f\rangle_m
+ \langle f-\mathcal{L}_nf,f-\mathcal{L}_nf\rangle_m\\
=&\langle f,f\rangle_m + \langle \mathcal{L}_nf,\mathcal{L}_nf\rangle_m .
\end{split}\end{equation}
Recalling the expression \eqref{equ:sigma} of
$$\sigma_{n,k,f}=\langle\mathcal{L}_nf-\mathcal{L}_kf,\mathcal{L}_nf-\mathcal{L}_kf\rangle-\langle\mathcal{L}_nf-\mathcal{L}_kf,\mathcal{L}_nf-\mathcal{L}_kf\rangle_m$$
and the observation \eqref{equ:observation2}, we have
\begin{equation*}\begin{split}
\langle f,(\mathcal{L}_n-\mathcal{L}_k)f\rangle_m =& \langle\mathcal{L}_nf-\mathcal{L}_kf,\mathcal{L}_nf-\mathcal{L}_kf\rangle\\
 = &\langle\mathcal{L}_nf-\mathcal{L}_kf,\mathcal{L}_nf-\mathcal{L}_kf\rangle_m+\sigma_{n,k,f}.
\end{split}\end{equation*}
Together with statement (c) of Lemma \ref{lem}, we have
\begin{equation*}\begin{split}
\langle \mathcal{L}_nf,\mathcal{L}_nf\rangle_m  &= \langle \mathcal{L}_kf,\mathcal{L}_kf\rangle_m+\langle\mathcal{L}_nf-\mathcal{L}_kf,\mathcal{L}_nf-\mathcal{L}_kf\rangle_m \\
&= \langle \mathcal{L}_kf,\mathcal{L}_kf\rangle_m + \langle f,(\mathcal{L}_n-\mathcal{L}_k)f\rangle_m - \sigma_{n,k,f}.
\end{split}\end{equation*}
Thus, replacing a sum of $\langle \mathcal{L}_kf,\mathcal{L}_kf\rangle_m + \langle f,(\mathcal{L}_n-\mathcal{L}_k)f\rangle_m$ on the left-hand side of \eqref{equ:summingup} with $\langle \mathcal{L}_nf,\mathcal{L}_nf\rangle_m + \sigma_{n,k,f}$ gives
\begin{equation}\label{equ:equality}
\langle \mathcal{L}_kf,\mathcal{L}_kf\rangle_m+ \langle f,(\mathcal{L}_n-\mathcal{L}_k)f\rangle_m + \sigma_{n,k,f} + \langle f-\mathcal{L}_nf,f-\mathcal{L}_nf\rangle_m
=\langle f,f\rangle_m.
\end{equation}
As $\sigma_{n,k,f}$ stands for the error in evaluating the integral of $(\mathcal{L}_nf-\mathcal{L}_kf)^2$ over $\Omega$ by the quadrature rule \eqref{equ:quad} with exactness degree $n+k$, the Marcinkiewicz--Zygmund property \eqref{equ:etaassumption} implies
$$|\sigma_{n,k,f}|\leq \eta\langle\mathcal{L}_nf-\mathcal{L}_kf,\mathcal{L}_nf-\mathcal{L}_kf\rangle =\eta \langle f,(\mathcal{L}_n-\mathcal{L}_k)f\rangle_m.$$
Thus,together with the non-negativeness of $\langle f-\mathcal{L}_nf,f-\mathcal{L}_nf\rangle_m$, the expression \eqref{equ:equality} leads to
$$\langle \mathcal{L}_kf,\mathcal{L}_kf\rangle_m + (1-\eta) \langle f,(\mathcal{L}_n-\mathcal{L}_k)f\rangle_m \leq\langle f,f\rangle_m, $$
that is,
$$\langle f,(\mathcal{L}_n-\mathcal{L}_k)f\rangle_m \leq \frac{1}{1-\eta}\left(\langle f,f\rangle_m-\langle \mathcal{L}_kf,\mathcal{L}_kf\rangle_m\right).$$
Hence, we have
\begin{equation*}
\begin{split}
\|\mathcal{L}_nf\|_2^2 =& \langle \mathcal{L}_kf,\mathcal{L}_kf\rangle_m + \langle f,(\mathcal{L}_n-\mathcal{L}_k)f\rangle_m\\
\leq& \frac{1}{1-\eta}\langle f,f\rangle_m-\frac{\eta}{1-\eta}  \langle \mathcal{L}_kf,\mathcal{L}_kf\rangle_m\\
\leq& \frac{1}{1-\eta}\langle f,f\rangle_m,
\end{split}\end{equation*}
and the stability result \eqref{equ:stability} follows from $$\langle f,f\rangle_m=\sum_{j=1}^mw_jf(x_j)^2\leq \sum_{j=1}^mw_j\|f\|_{\infty}^2=V\|f\|_{\infty}^2.$$

The error bound \eqref{equ:error} can be derived from a standard argument. For any $\chi\in\mathbb{P}_k$, with the aid of Lemma \ref{lem:polynomial}, there holds $\mathcal{L}_nf-f =\mathcal{L}_n(f-\chi)-(f-\chi)$. Using the stability result \eqref{equ:stability}, we have
\begin{equation*}\begin{split}
\|\mathcal{L}_nf-f\|_2
& = \|\mathcal{L}_n(f-\chi)-(f-\chi)\|_2\\
& \leq  \|\mathcal{L}_n(f-\chi)\|_2+\|f-\chi\|_2\\
&\leq \frac{V^{1/2}}{\sqrt{1-\eta}}\|f-\chi\|_{\infty} +V^{1/2}\|f-\chi\|_{\infty}\\
&=\left(\frac{1}{\sqrt{1-\eta}}+1\right) V^{1/2} \|f-\chi\|_{\infty}.
\end{split}\end{equation*}
This estimate implies, as it holds for all $\chi\in\mathbb{P}_k$, that
\begin{equation*}\begin{split}
\|\mathcal{L}_nf-f\|_2\leq &\left(\frac{1}{\sqrt{1-\eta}}+1\right) V^{1/2}\inf_{\chi\in\mathbb{P}_k}\|f-\chi\|_{\infty}\\
=&\left(\frac{1}{\sqrt{1-\eta}}+1\right)V^{1/2}E_k(f).
\end{split}\end{equation*}
If $k$ is fixed, then $E_k(f)$ is fixed, suggesting that no convergence result of $\mathcal{L}_nf$ as $n\rightarrow\infty$ can be concluded. On the other hand, if $k$ is positively correlated to $n$, then $E_k(f)\rightarrow0$ and hence $\|\mathcal{L}_nf-f\|_2\rightarrow 0$ as $n\rightarrow \infty$. \hfill $\square$

\section{Numerical examples}\label{sec:3}

We now apply Theorem \ref{thm} to two regions: the interval $[-1,1]\subset\mathbb{R}$ and the $2$-sphere $\mathbb{S}^2\subset\mathbb{R}^3$. For the simplicity of the narrative, we assume that the following mentioned quadrature rules have the Marcinkiewicz--Zygmund property \eqref{equ:etaassumption} with $\eta= 3/4$, a quite loose assumption for $\eta\in[0,1)$. All codes were written by MATLAB R2022a, and all numerical experiments were conducted on a laptop (16 GB RAM, Intel® CoreTM i7-9750H Processor) with macOS Monterey 12.4. The codes are available at https://github.com/HaoNingWu/MZHyper.

\subsection{The interval}
Let $\Omega  = [ -1,1]$ with $\text{d}\omega  = \omega (x)\text{d}x$, where $\omega(x) \geq  0$
is a weight function on $[- 1,1]$ and different $\omega (x)$ leads to different value of $V =
\int_{-1}^1 \omega (x)\text{d}x$. The space $\mathbb{P}_n$ is a linear space of polynomials of degree at most $n$ on $[-1,1]$, hence $d_n = n+1$.

In the following example, we consider $\omega(x) = 1$ (thus $V=2$), and quadrature rules with such weight function include the Gauss--Legendre quadrature and the Clenshaw--Curtis quadrature. We refer the reader to \cite{trefethen2022exactness} for background information about quadrature rules on $[-1,1]$. The Gauss--Legendre quadrature rule is a typical choice of quadrature rules for the original hyperinterpolation $\mathcal{L}^{\text{S}}_n$, as an $m$-point Gauss--Legendre quadrature has exactness degree $2m-1$.  For effective testing of Gaussian quadrature rules, we refer the reader to \cite{gautschi1983and}. Thus, an $(n+1)$-point Gauss--Legendre quadrature can fulfill the exactness requirement $2n$ of $\mathcal{L}^{\text{S}}_n$. Meanwhile, the Clenshaw--Curtis quadrature \cite{Clenshaw1960} in the Chebyshev points, which has exactness degree $m-1$ if $m$ quadrature points are adopted, is not considered practical in constructing the original hyperinterpolants. Indeed, one needs a $(2n+1)$-point Clenshaw--Curtis quadrature to construct an original hyperinterpolant $\mathcal{L}^{\text{S}}_nf$. However, in the light of Theorem \ref{thm}, we have the following corollary.

\begin{corollary}
Let $\langle\cdot,\cdot\rangle_m$ used in Definition \ref{def} be an $m$-point Gauss-Legendre quadrature with $(n+2)/{2}\leq m\leq (2n+1)/{2}$, or an $m$-point Clenshaw--Curtis quadrature with $n+2\leq m\leq 2n+1$. Under the conditions of Theorem \ref{thm} with $\eta = 3/4$, the exactness-relaxing hyperinterpolant $\mathcal{L}_nf$ satisfies
\begin{equation*}
\|\mathcal{L}_nf-f\|_2\leq
\begin{cases}
3V^{1/2}E_{2m-1-n}(f)&\text{when using the Gauss--Legendre quadrature},\\
3V^{1/2}E_{m-1-n}(f)&\text{when using the Clenshaw--Curtis quadrature}.
\end{cases}
\end{equation*}
\end{corollary}

It is worth noting that the $m$-point Newton--Cotes quadrature in the equispaced points with $n+2\leq m\leq 2n+1$, though having exactness degree exceeding $n+1$, fails to fulfill the assumption of positive weights, as the Newton--Cotes weights have alternating signs. However, this does not suggest the impossibility of constructing hyperinterpolants in the equispaced points. Quadrature rules with exactness $n+k$ in the equispaced points, even in the scattered points, can be designed in the spirit of optimal recovery rather than the exactness principle. As suggested in \cite{devore2019computing}, given $m$ distinct points $\{x_j\}_{j=1}^m$, one can design a quadrature with exactness degree $n+k$ by obtaining its quadrature weights $\{w_j\}_{j=1}^m$ from solving
\begin{equation}\label{equ:optimalrecovery}
\min\limits_{w_1,w_2,\ldots,w_m}\sum_{j=1}^m|w_j| \quad\text{s.t.}\quad\sum_{j=1}^mw_j v(x_j)=\int_{-1}^1v\quad\forall v\in\mathbb{P}_{n+k}.
\end{equation}
In general, the number $m$ of quadrature points in the rule \eqref{equ:optimalrecovery} should be much larger than the exactness-oriented quadrature rules to achieve the exactness degree $n+k$. For example, to design an $m$-equispaced-point quadrature with exactness degree $n+k$ in the spirit of \eqref{equ:optimalrecovery}, $m$, $n$, and $k$ shall satisfy $n+k=\mathcal{O}(\sqrt{m\ln{m}})$, see \cite[Theorem 3.6]{devore2019computing}. Thus, we have the following result.

\begin{corollary}
Let $\langle\cdot,\cdot\rangle_m$ used in Definition \ref{def} be an $m$-point quadrature designed by \eqref{equ:optimalrecovery}, where the quadrature points are equispaced points on $[-1,1]$, and the weights should be positive. Under the conditions of Theorem \ref{thm} with $\eta=3/4$, the error of the exactness-relaxing hyperinterpolant $\mathcal{L}_nf$ is controlled by $\|\mathcal{L}_nf-f\|_2\leq 3V^{1/2}E_{k}(f)$.
\end{corollary}

We present a toy example on the interval $[-1,1]$ to illustrate Theorem \ref{thm} on $\Omega=[-1,1]$. We are interested in a 40-degree hyperinterpolant $\mathcal{L}_{40}f$ of $f=\exp(-x^2)$ and $f=|x|^{5/2}$, with $\{p_{\ell}\}_{\ell=1}^{41}$ chosen as normalized Legendre polynomials $\{P_{\ell}\}_{\ell=0}^{40}$. The former test function $f=\exp(-x^2)$ is an analytic function (so smooth enough) and the latter $f= |x|^{5/2}$ is only continuous (not even differentiable).

Constructing $\mathcal{L}^{\text{S}}_{40}f$ requires a quadrature rule with exactness degree 80, thus one may consider a 41-point Gauss quadrature with exactness degree 81. Besides, we also construct $\mathcal{L}_{40}f$ using a 25-point Gauss-Legendre quadrature, a 50-point Clenshaw--Curtis quadrature, and a 186-point quadrature \eqref{equ:optimalrecovery} in equispaced points with exactness degree 49. These quadrature rules all have the exactness degree 49, which is far from the required degree 80 for $\mathcal{L}^{\text{S}}_{40}f$, but they also enable us to obtain hyperinterpolants with considerably small errors. On the other hand, the relaxation of quadrature exactness, suggested in Theorem \ref{thm}, slows the convergence rates of hyperinterpolants. That is, the $L^2$ error estimation of $\mathcal{L}^{\text{S}}_{40}f$ is controlled by $E_{40}(f)$, suggested by the estimation \eqref{equ:errororiginal} derived in Sloan's original work \cite{sloan1995polynomial}, while that of $\mathcal{L}_{40}f$ is controlled by $E_9(f)$, according to our error estimation \eqref{equ:error}.

The performance of $\mathcal{L}_{40}f$ in the approximation of both functions is displayed in Figures \ref{fig:interval1} and \ref{fig:interval2}, respectively. Our theoretical analysis of the effects of the relaxing quadrature exactness is also verified in both figures. Besides, the numerical results suggest that such effects may also be related to the smoothness of functions to be approximated. That is, the error of $\mathcal{L}^{\text{S}}_{40}f$ is much smaller than the errors of $\mathcal{L}_{40}f$ using three different quadrature rules for the analytic function $f=\exp(-x^2)$, but just slightly smaller than those for the non-differentiable function $f=|x|^{5/2}$. Moreover, it is pretty interesting that the hyperinterpolant $\mathcal{L}_{40}f$ with the 50-point Clenshaw--Curtis quadrature performs better than that using the 25-point Gauss--Legendre quadrature and the 186-point quadrature \eqref{equ:optimalrecovery} in equispaced points, though three quadrature rules have the same exactness degree $49$. This finding is worthy of further study. To the authors' best knowledge, the connection between the Clenshaw--Curtis quadrature and the performance of hyperinterpolation has not been established. Some possibly useful results that help us to establish such a connection can be found in Trefethen's famous paper \cite{trefethen2008gauss}.

\begin{figure}[htbp]
  \centering
  \includegraphics[width=\textwidth]{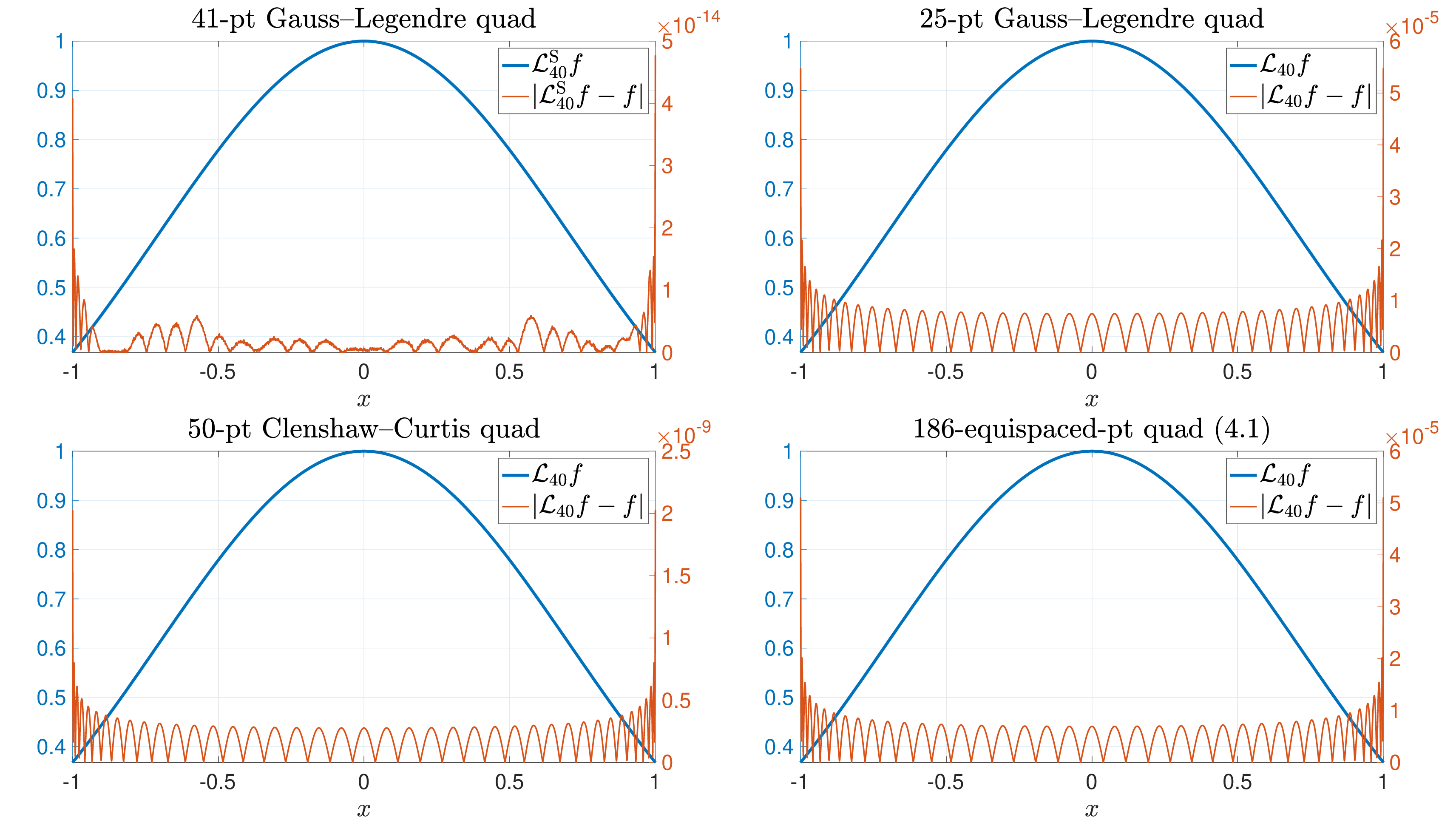}
  \caption{Hyperinterpolants $\mathcal{L}^{\text{S}}_{40}f$ and $\mathcal{L}_{40}f$ of $f=\exp(-x^2)$, constructed by various quadrature rules. The estimation of $\|\mathcal{L}^{\text{S}}_{40}f-f\|_2$ is controlled by $E_{40}(f)$, while that of $\|\mathcal{L}_{40}f-f\|_2$ by $E_{9}(f)$.}\label{fig:interval1}
\end{figure}

\begin{figure}[htbp]
  \centering
  \includegraphics[width=\textwidth]{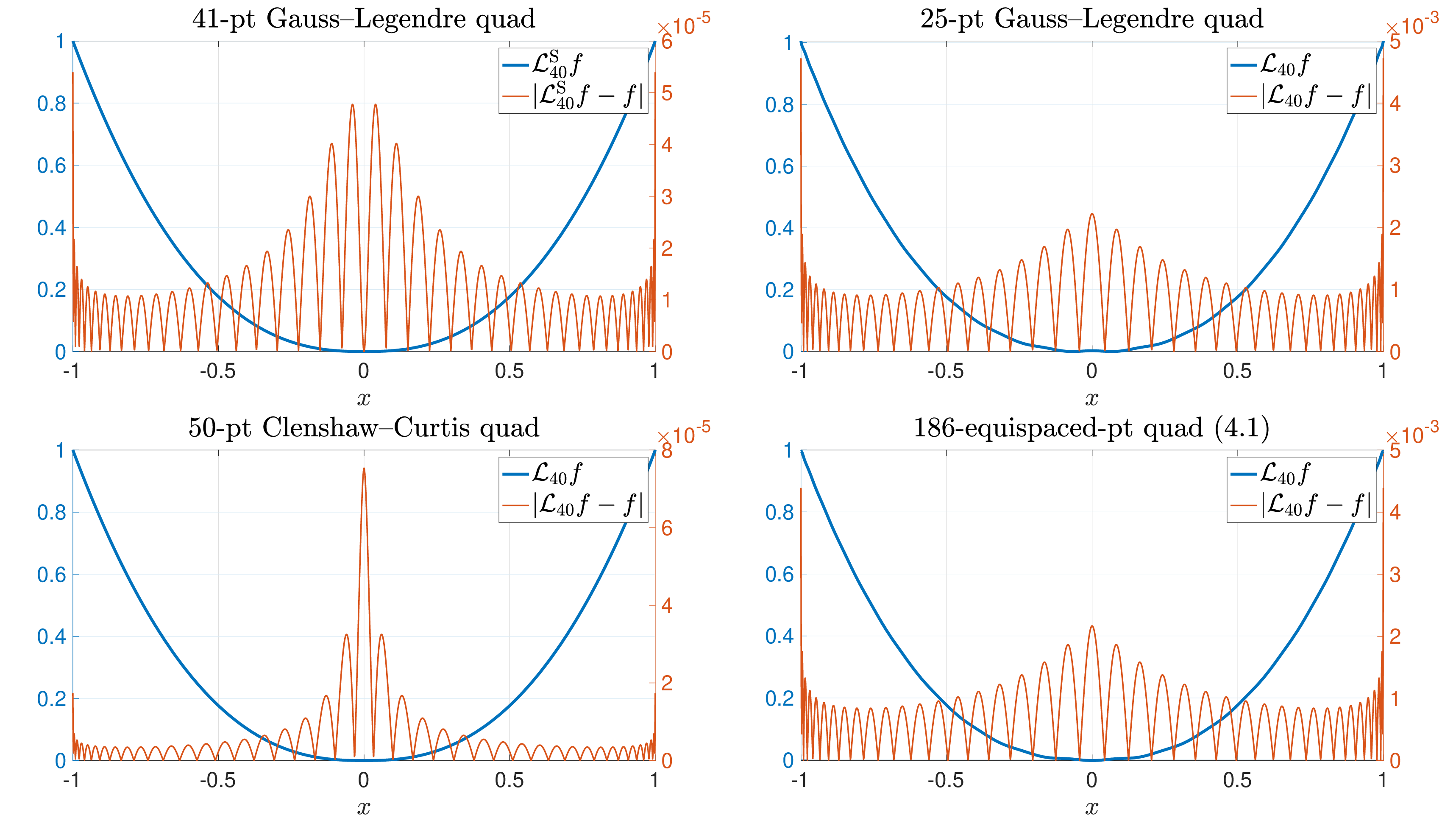}
  \caption{Hyperinterpolants $\mathcal{L}^{\text{S}}_{40}f$ and $\mathcal{L}_{40}f$ of $f=|x|^{5/2}$, constructed by various quadrature rules. The estimation of $\|\mathcal{L}^{\text{S}}_{40}f-f\|_2$ is controlled by $E_{40}(f)$, while that of $\|\mathcal{L}_{40}f-f\|_2$ by $E_{9}(f)$.}\label{fig:interval2}
\end{figure}

\subsection{The sphere}
Let $\Omega = \mathbb{S}^2\subset\mathbb{R}^3$ with $\text{d}\omega  = \omega (x)\text{d}x$, where $\omega(x)$ is an area measure on $\mathbb{S}^2$. Thus $V=\int_{\mathbb{S}^2}\text{d}\omega  = 4\pi$ denotes the surface area of $\mathbb{S}^2$. In this example, $\mathbb{P}_n$ can be regarded as the space of spherical polynomials of degree at most $n$. Let the basis $\{p_{\ell}\}_{\ell=1}^{d_n}$ be a set of orthonormal spherical harmonics $ \{ Y_{\ell,k}: \ell  = 0, 1,\ldots, n, k = 1,\ldots,2\ell +1\}$ , and the dimension of $\mathbb{P}_n$ is $d_n=(n+1)^2$. Many positive-weight quadrature rules can achieve the desired exactness degree, such as rules using spherical $t$-designs \cite{delsarte1991geometriae} and tensor-product quadrature rules from rules on the interval \cite{zbMATH01421286}, which are both designed on structural quadrature points. Thanks to the work of Mhaskar, Narcowich, and Ward \cite{mhaskar2001spherical}, it was also proved that positive-weight quadrature rules with desired polynomial exactness could be designed from scattered data. All of these rules requires $m=\mathcal{O}(k^2)$ points to achieve the exactness degree $k$. Thus roughly speaking, to construct an original hyperinterpolant requires $4cn^2$ points, where $c>0$ is some constant, while in the light of Theorem \ref{thm}, only $c(n+k)^2$ points with $0<k\leq n$ are needed.

For the sake of easy implementation, we discuss Theorem \ref{thm} with quadrature rules using spherical $t$-designs, which can be implemented easily and efficiently. A point set $\{x_1,x_2,\ldots,x_m\}\subset\mathbb{S}^2$ is said to be a \emph{spherical $t$-design} \cite{delsarte1991geometriae} if it satisfies
\begin{equation}\label{equ:sphericalquadrature}
\frac{1}{m}\sum_{j=1}^mv(x_j)=\frac{1}{4\pi}\int_{\mathbb{S}^2}v\text{d}\omega\quad\forall v\in\mathbb{P}_t.
\end{equation}
It can be seen that spherical $t$-design is  a set of points on the sphere such that an equal-weight quadrature rule in these points integrates all (spherical) polynomials up to degree $t$ exactly. In this paper, we employ well conditioned spherical $t$-designs \cite{MR2763659}, which are suitable for numerical integration and interpolation. The study in \cite{an2012regularized} revealed that well conditioned spherical $t$-designs can be used to realize  hyperinterpolation and regularization approximation  successfully. Well conditioned spherical $t$-designs require at least $(t+1)^2$ quadrature points to achieve the exactness degree $t$ \cite{MR2763659}. Thus, it requires at least $(2n+1)^2$ points to construct an original hyperinterpolant of degree $n$. However, thanks to Theorem \ref{thm}, we have the following result.

\begin{corollary}\label{cor:sphere}
Let $\langle\cdot,\cdot\rangle_m$ used in Definition \ref{def} be the quadrature rule \eqref{equ:sphericalquadrature} using a spherical $(n+k)$-design with $0<k\leq n$. The number $m$ of quadrature points should satisfy $m\geq (n+k+1)^2$. Under the conditions of Theorem \ref{thm} with $\eta = 3/4$, the exactness-relaxing hyperinterpolant $\mathcal{L}_nf$ satisfies
\begin{equation*}
\|\mathcal{L}_nf-f\|_2\leq 6\pi^{1/2}E_{k}(f).
\end{equation*}
In particular, if the spherical $(n+k)$-design with $m = (n+k+1)^2$ is used, then
\begin{equation*}
\|\mathcal{L}_nf-f\|_2\leq 6\pi^{1/2}E_{\sqrt{m}-n-1}(f).
\end{equation*}
\end{corollary}


We present a toy illustration on the sphere, making use of the well conditioned spherical $t$-designs \cite{MR2763659} with $m = (t+1)^2$. We are interested in a 25-degree hyperinterpolant $\mathcal{L}_{25}f$ of a Wendland function $f$: Let $\mathbf{z}_1=[1,0,0]^{\rm{T}}$, $\mathbf{z}_2=[-1,0,0]^{\rm{T}}$, $\mathbf{z}_3=[0,1,0]^{\rm{T}}$, $\mathbf{z}_4=[0,-1,0]^{\rm{T}}$, $\mathbf{z}_5=[0,0,1]^{\rm{T}}$, and $\mathbf{z}_6=[0,0,-1]^{\rm{T}}$, the testing function $f$ is defined as
\begin{equation}\label{equ:wendland}
f(\mathbf{x})=\sum_{i=1}^6\phi_2(\|\mathbf{z}_i-\mathbf{x}\|_2),
\end{equation}
where $\phi_2(r):=\tilde{\phi}_2\left(r/\delta_2\right)$ is a normalized Wendland function \cite{chernih2014wendland}, with
$$\tilde{\phi}_2(r):=\left(\max\{1-r,0\}\right)^6(35r^2+18r +3)/3$$
been an original Wendland function \cite{wendland1995piecewise} and $\delta_2=(9\Gamma(5/2))/(2\Gamma(3))$. According to the original definition of hyperinterpolation \eqref{equ:hyperinterpolation}, one shall use a spherical $50$-design and its corresponding quadrature rule to construct $\mathcal{L}_{25}^{\text{S}}f$. To tell the difference between $\mathcal{L}_{25}^{\text{S}}f$ and $\mathcal{L}_{25}f$, we also use a sphere $30$-design and its corresponding quadrature rule to construct $\mathcal{L}_{25}f$. Both designs are displayed in Figure \ref{fig:sphericaldesigns}.

\begin{figure}[htbp]
  \centering
  \includegraphics[width=0.78\textwidth]{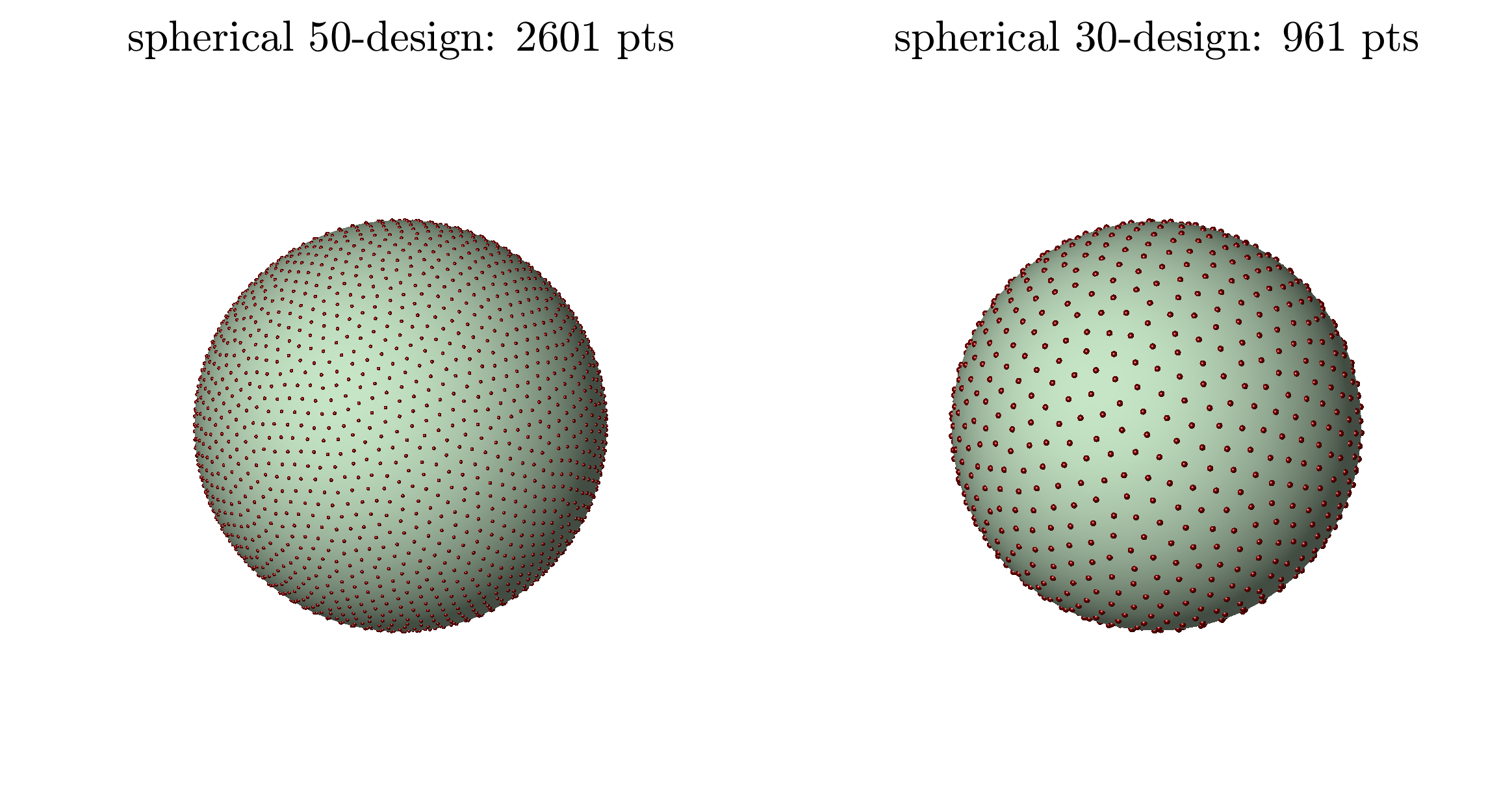}
  \caption{Spherical $50$- and $30$-designs, generated by the method proposed in \cite{MR2763659}.}\label{fig:sphericaldesigns}
\end{figure}

The original hyperinterpolant $\mathcal{L}^{\rm{S}}_{25}f$ of the Wendland-type function \eqref{equ:wendland} and the corresponding error are plotted in the upper row of Figure \ref{fig:sphere1}. According to Sloan \cite{sloan1995polynomial}, the $L^2$ error estimation of $\mathcal{L}^{\rm{S}}_{25}f$ is controlled by $E_{25}(f)$. Corollary \ref{cor:sphere} indicates that $\mathcal{L}_{25}f$ can be obtained using an exactness-relaxing quadrature rule. This is shown in the lower row in Figure \ref{fig:sphere1}, where a sphere $30$-design and its corresponding quadrature rule are used. Corollary \ref{cor:sphere} also suggests that the $L^2$ error estimation of $\mathcal{L}_{25}f$ is thus controlled by $E_{5}(f)$.

\begin{figure}[htbp]
  \centering
  \includegraphics[width=\textwidth]{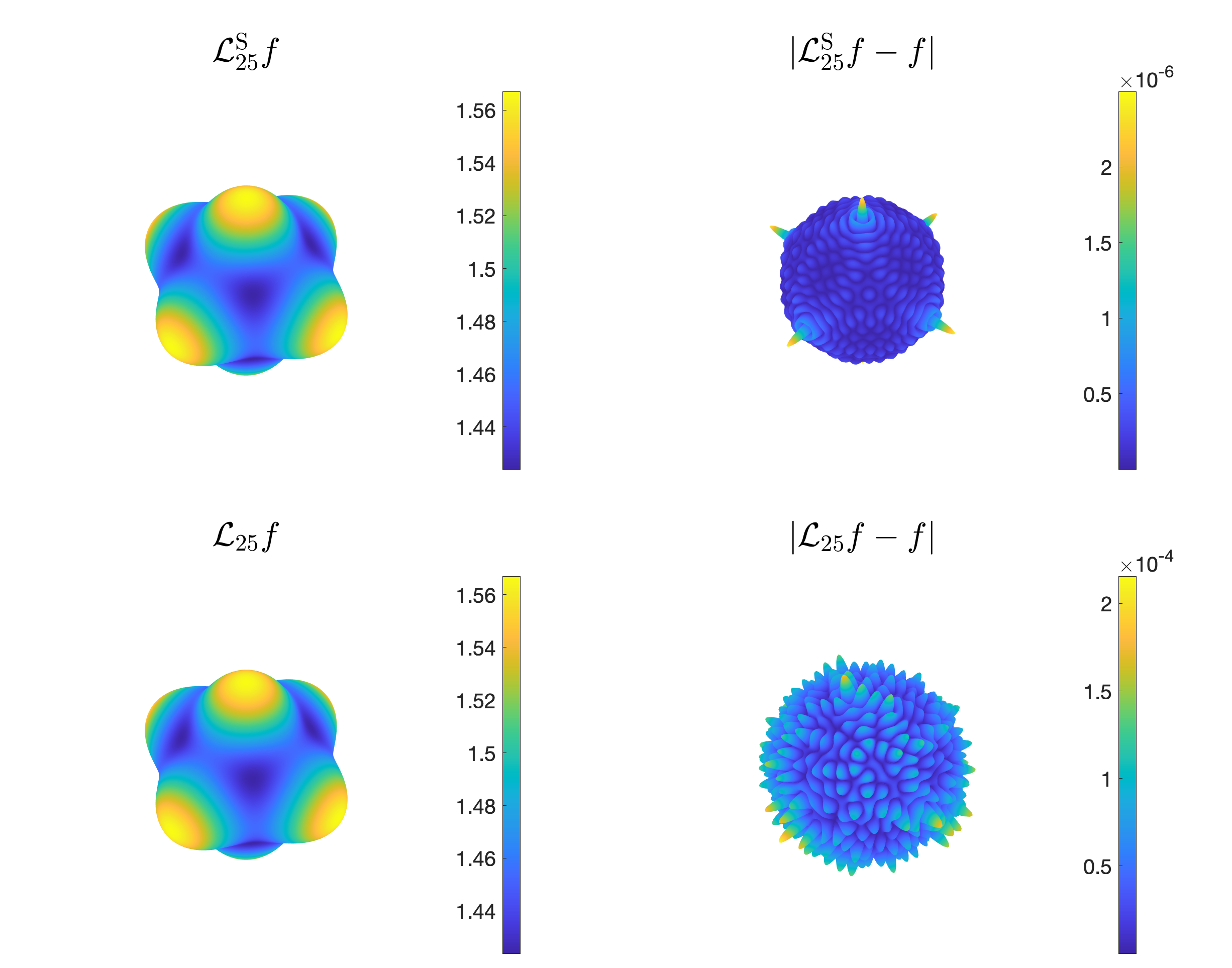}
  \caption{Hyperinterpolants $\mathcal{L}^{\rm{S}}_{25}f$ and $\mathcal{L}_{25}f$ of a Wendland-type function \eqref{equ:wendland}, constructed by spherical $t$-designs with $t=50$ (upper row) and $30$ (lower row), respectively. The estimation of $\|\mathcal{L}^{\text{S}}_{25}f-f\|_2$ is controlled by $E_{25}(f)$, while that of $\|\mathcal{L}_{25}f-f\|_2$ by $E_{5}(f)$.}\label{fig:sphere1}
\end{figure}

Along with the Wendland-type function \eqref{equ:wendland}, we additionally test the function $f(\mathbf{x})=f(x,y,z)=|x+y+z|$
with $\mathbf{x}=[x,y,z]^{\rm{T}}\in\mathbb{S}^2$. Similar to the above test, the original hyperinterpolant $\mathcal{L}^{\rm{S}}_{25}f$ and the corresponding error are plotted in the upper row of Figure \ref{fig:sphere2}, and the hyperinterpolant $\mathcal{L}_{25}f$ and its error are shown in the lower row of Figure \ref{fig:sphere2}. This test also validates our theory on the effects of the relaxing quadrature exactness. Moreover, as the function $f(x,y,z)=|x+y+z|$ is not differentiable, similar to the non-differentiable function $f(x)=|x|^{5/2}$ on $[-1,1]$, we see than the error of $\mathcal{L}^{\rm{S}}_{25}f$ is just slightly smaller than that of $\mathcal{L}_{25}f$.

\begin{figure}[htbp]
  \centering
  \includegraphics[width=\textwidth]{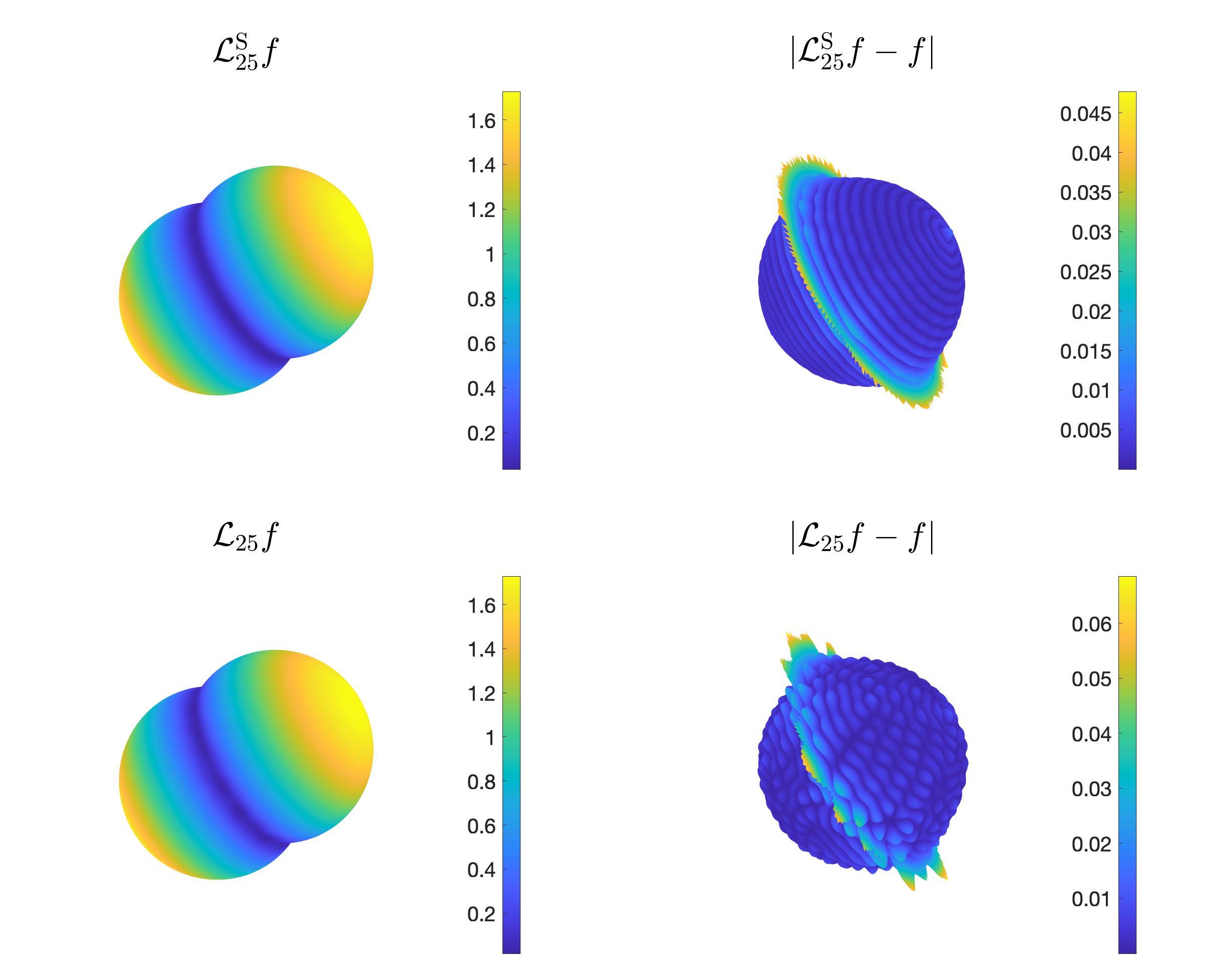}
  \caption{Hyperinterpolants $\mathcal{L}^{\rm{S}}_{25}f$ and $\mathcal{L}_{25}f$ of $f(\mathbf{x})=f(x,y,z)=|x+y+z|$, constructed by spherical $t$-designs with $t=50$ (upper row) and $30$ (lower row), respectively. The estimation of $\|\mathcal{L}^{\text{S}}_{25}f-f\|_2$ is controlled by $E_{25}(f)$, while that of $\|\mathcal{L}_{25}f-f\|_2$ by $E_{5}(f)$.}\label{fig:sphere2}
\end{figure}

We close this paper with a more detailed study on the error behavior of the exactness-relaxing hyperinterpolation on the sphere. In the above tests, we let $k=25$ (for constructing $\mathcal{L}^{\text{S}}_{25}f$) and $5$ (for constructing $\mathcal{L}_{25}f$). Now letting $k$ range from $1$ to $25$, that is, the exactness from $26$ to $50$, the $L^2$ and uniform errors of $\mathcal{L}_{25}f$ in the approximation of both functions are displayed in Table \ref{tab:sphere}. We see from Table \ref{tab:sphere} that, in general, the errors $\|\mathcal{L}_nf-f\|_2$ and $\|\mathcal{L}_nf-f\|_{\infty}$ reduce as $k$ increases. This behavior of $\|\mathcal{L}_nf-f\|_2$ is predicted by our theory: Theorem \ref{thm} indicates that the $L^2$ error of $\mathcal{L}_{n}f$ using the quadrature rule with exactness $n+k$ is controlled by $E_k(f)$.

\begin{table}[htbp]
  \centering
  \setlength{\abovecaptionskip}{0pt}
\setlength{\belowcaptionskip}{10pt}
  \caption{Performance of the exactness-relaxing hyperinterpolation: $n=25$ and $k$ ranges from $1$ to $25$.}\label{tab:sphere}
  \ttfamily
    \begin{tabular}{|c||c|c||c|c|}
   \hline
    ~ & \multicolumn{2}{c||}{$n=25$} & \multicolumn{2}{c|}{$n=25$}\\
    ~ & \multicolumn{2}{c||}{$f$: Wendland function \eqref{equ:wendland}} & \multicolumn{2}{c|}{$f(x,y,z)=|x+y+z|$} \\ \hline
          $(k,n+k,m)$ & $\|\mathcal{L}_nf-f\|_2$ & $\|\mathcal{L}_nf-f\|_{\infty}$ & $\|\mathcal{L}_nf-f\|_2$ & $\|\mathcal{L}_nf-f\|_{\infty}$ \\ \hline
(1,26,729) & 1.4703e-04 & 1.1973e-02 & 1.3806e-03 & 1.4653e-01  \\\hline
(2,27,784) & 1.0036e-04 & 7.2393e-03 & 5.9539e-04 & 7.6914e-02  \\\hline
(3,28,841) & 7.7225e-05 & 5.6280e-03 & 5.1663e-04 & 9.2067e-02  \\\hline
(4,29,900) & 3.6550e-06 & 2.2721e-04 & 4.7716e-04 & 6.3882e-02  \\\hline
(5,30,961) & 2.7813e-06 & 2.1562e-04 & 4.3549e-04 & 6.8573e-02  \\\hline
(6,31,1024) & 9.0144e-07 & 7.3522e-05 & 4.1188e-04 & 6.7465e-02  \\\hline
(7,32,1089) & 6.3510e-07 & 5.4311e-05 & 4.1158e-04 & 6.9123e-02  \\\hline
(8,33,1156) & 1.5667e-07 & 1.4221e-05 & 3.8191e-04 & 5.7172e-02  \\\hline
(9,34,1225) & 1.2137e-07 & 1.0454e-05 & 3.7573e-04 & 5.7909e-02  \\\hline
(10,35,1296) & 6.0979e-08 & 7.9442e-06 & 3.7698e-04 & 5.7189e-02  \\\hline
(11,36,1369) & 5.3640e-08 & 5.4959e-06 & 3.7237e-04 & 6.0998e-02  \\\hline
(12,37,1444) & 1.8896e-08 & 3.3341e-06 & 3.6456e-04 & 5.6171e-02  \\\hline
(13,38,1521) & 1.9095e-08 & 3.7055e-06 & 3.6651e-04 & 5.5231e-02  \\\hline
(14,39,1600) & 1.6651e-08 & 3.2061e-06 & 3.6385e-04 & 5.3134e-02  \\\hline
(15,40,1681) & 1.4991e-08 & 2.6047e-06 & 3.5941e-04 & 5.2498e-02  \\\hline
(16,41,1764) & 1.4137e-08 & 2.9486e-06 & 3.6263e-04 & 5.2798e-02  \\\hline
(17,42,1849) & 1.3659e-08 & 2.5557e-06 & 3.5752e-04 & 5.0185e-02  \\\hline
(18,43,1936) & 1.3509e-08 & 2.5579e-06 & 3.5447e-04 & 5.0666e-02  \\\hline
(19,44,2025) & 1.3433e-08 & 2.5896e-06 & 3.5454e-04 & 5.0915e-02  \\\hline
(20,45,2116) & 1.3354e-08 & 2.6336e-06 & 3.5534e-04 & 5.0098e-02  \\\hline
(21,46,2209) & 1.3318e-08 & 2.5630e-06 & 3.5320e-04 & 4.8124e-02  \\\hline
(22,47,2304) & 1.3309e-08 & 2.4906e-06 & 3.5443e-04 & 5.0818e-02  \\\hline
(23,48,2401) & 1.3309e-08 & 2.5130e-06 & 3.5375e-04 & 4.7735e-02  \\\hline
(24,49,2500) & 1.3294e-08 & 2.4568e-06 & 3.5180e-04 & 4.8141e-02  \\\hline
(25,50,2601) & 1.3294e-08 & 2.4959e-06 & 3.5146e-04 & 4.7660e-02  \\\hline
  \end{tabular}
  \end{table}

 \section*{Acknowledgements}
We thank both referees for their valuable suggestions and remarks
which improved this paper.
\bibliographystyle{siam}

\bibliography{myref}   

\end{document}